\documentclass[11pt]{article}
\usepackage[T1]{fontenc}
\usepackage{geometry}
\usepackage{color}
\usepackage{float}
\usepackage{amsmath}
\usepackage{amssymb}
\usepackage{amsthm}
\usepackage{bbold}
\usepackage{mathrsfs}
\usepackage{setspace}
\usepackage[all,cmtip]{xy}
\usepackage{tikz}
\usepackage{enumerate}
\usepackage{mathtools}
\usepackage{lipsum}

\newtagform{roman}[]()

\newcommand{\Z}{{\mathbb{Z}}}
\newcommand{\Q}{{\mathbb{Q}}}
\newcommand{\R}{{\mathbb{R}}}

\newcommand{\cO}{{\mathscr O}}

\newcommand{\cF}{{\mathscr F}}
\newcommand{\cC}{{\mathscr C}}
\newcommand{\cL}{{\mathscr L}}
\newcommand{\ep}{{\varepsilon}}
\newcommand{\GL}{{\mathrm {GL}}}
\newcommand{\Gal}{{\mathrm {Gal}}}
\newcommand{\Id}{{\mathrm {Id}}}

\newcommand{\Cl}{{\mathrm {Cl}}}

\newtheorem{thm}{Theorem}[section]
\newtheorem{prop}[thm]{Proposition}

\newtheorem{lem}[thm]{Lemma}
\newtheorem{rem}[thm]{Remark}
\newtheorem{cond}[thm]{Condition}
\theoremstyle{definition}

\numberwithin{equation}{section}

\begin{document}

\large{\centerline{\textbf{$\ell$-TORSION IN CLASS GROUPS}}}

\large{\centerline{\textbf{OF CERTAIN FAMILIES OF $D_4$-QUARTIC FIELDS}}}

\vspace{5mm}

\large{\centerline{Chen An}}
\vspace{3mm}
\normalsize{}
\begin{abstract}
We prove an upper bound for $\ell$-torsion in class groups of almost all fields in certain families of $D_4$-quartic fields. Our key tools are a new Chebotarev density theorem for these families of $D_4$-quartic fields and a lower bound for the number of fields in the families.
\end{abstract}

\begin{section}{Introduction}

In this paper we prove the first unconditional nontrivial upper bound on $\ell$-torsion, for all positive integers $\ell \ge 1$, in class groups of certain $D_4$-quartic fields. In particular, this holds for almost all fields in any infinite family of $D_4$-quartic fields associated to a fixed biquadratic field.

The ideal class group $\Cl_K$, defined for every number field $K$, is the quotient group of the fractional ideals modulo principal ideals. For an integer $\ell \ge 1$, we define the $\ell$-torsion subgroup
\begin{equation*}
\Cl_K[\ell]=\{ [\mathfrak{a} ] \in \Cl_K: [\mathfrak{a} ]^\ell=\Id \}
\end{equation*}
and let $d=[K:\Q]$. We denote $D_K$ as the absolute value of $\mathrm{disc}(K/\Q)$. Then for any $\ep>0$, one has the trivial bound 
$|\Cl_K[\ell]| \le |\Cl_K| \ll_{d,\ep} D_K^{1/2+\ep}$.
But it is widely conjectured that
$|\Cl_K[\ell]| \ll_{d,\ell,\ep} D_K^{\ep}$, for any $\ep>0$. Progress towards this has been difficult. Even under GRH, one only obtains
\begin{equation}\label{grh}
|\Cl_K[\ell]| \ll_{d,\ell,\ep} D_K^{\frac12-\frac{1}{2\ell(d-1)}+\ep},
\end{equation}
for all $\ep>0$; see Proposition 3.1 of \cite{EV07}.

In the recent work of \cite{HBP17}, \cite{EPW17}, \cite{PTW17}, \cite{FW17}, \cite{Wid17}, \cite{TZ19}, nontrivial upper bounds for $\ell$-torsion at least as strong as (\ref{grh}) have been proved for almost all fields in certain families of degree $d$ fields, for any $d \ge 2$, but notably $D_4$-quartic fields have not been treated in these works. This omission motivates the work of this paper, which exhibits an infinite collection of families of $D_4$-quartic fields for which we can prove such bounds. See also \cite{FW18} and \cite{PTW19} for results on $\ell$-torsion bounds from the perspective of moments.

For a number field $K$, let $\widetilde{K}$ be the Galois closure of $K$ over $\Q$ within a fixed choice of $\overline{\Q}$. By a $D_4$-quartic field $K$ we mean a quartic extension $K$ of $\Q$ such that $\mathrm{Gal}(\widetilde{K}/\Q) \cong D_4$, and we will define our families of $D_4$-quartic fields according to a fixed biquadratic extension of $\Q$. We write $Q=\Q(\sqrt{a},\sqrt{b})$ as a biquadratic field over $\Q$, where $a,b$ are distinct square-free integers not equal to $0$ or $1$. Denoting $\xi=\mathrm{gcd}(|a|,|b|)$, we have $Q=\mathrm{span}_{\Q} \{1,\sqrt{a},\sqrt{b},\frac{\sqrt{ab}}{\xi} \}$. For any such $Q$, we define the family 
\begin{equation*}
\cF_4(Q)=\{K: K \text{ is a } D_4\text{-quartic field}, \widetilde{K} \text{ contains } Q=\Q(\sqrt{a},\sqrt{b})\}
\end{equation*}
and denote
\begin{equation*}\label{f4}
\cF_4(Q;X)=\{K: K \in \cF_4(Q), D_K \le X \}.
\end{equation*}

From the lattice of fields in Section \ref{mot}, we will see that for any $D_4$-quartic field $K$, $\widetilde{K}$ contains a unique biquadratic subextension $Q$. Therefore, taking all the families $\cF_4(Q)$ for $Q=\Q(\sqrt{a},\sqrt{b})$ as $a$ and $b$ vary, we obtain all $D_4$-quartic fields. In other words, taken together, these families are ``generic''.

Our first main result of this paper is the following theorem on bounding $\ell$-torsion in class groups of almost all fields in $\cF_4(Q;X)$ for each choice of a biquadratic field $Q$ such that $\cF_4(Q) \neq \emptyset$.
\begin{thm}\label{tor}
Let $Q=\Q(\sqrt{a},\sqrt{b})$ be such that $\cF_4(Q) \neq \emptyset$. For every $0<\ep<\frac14$ sufficiently small, and every integer $\ell \ge 1$, there exists a parameter $B_1=B_1(\ell,\ep)$ such that for every $X \ge 1$, aside from at most $B_1X^{\ep}$ fields in $\cF_4(Q;X)$, every field $K \in \cF_4(Q;X)$ satisfies
\begin{equation}\label{toreq}
|\Cl_K[\ell]| \ll_{\ell,\ep} D_K^{\frac12-\frac{1}{6\ell}+\ep}.
\end{equation}
\end{thm}

Theorem \ref{tor} provides the first unconditional nontrivial bound for $\ell$-torsion in class groups of infinite families of $D_4$-quartic fields.

In order to show that almost all fields in $\cF_4(Q)$ satisfy (\ref{toreq}), we must exhibit a lower bound for $|\cF_4(Q;X)|$ that grows strictly faster than $B_1X^{\ep}$. This leads to the following theorem, as our second main result.

\begin{thm}\label{vital}
Let $Q=\Q(\sqrt{a},\sqrt{b})$ be such that $\cF_4(Q) \neq \emptyset$. Then 
\begin{equation*}
X^{1/2} \ll_{Q} |\cF_4(Q;X)| \ll  X.
\end{equation*}
\end{thm} 

The lower bound in Theorem \ref{vital} is the first nontrivial lower bound for such families of $D_4$-quartic fields. The upper bound is an immediate consequence of the result $N_4(D_4,X) \sim c(D_4)X$ (where $c(D_4)>0$) in \cite{CDO02}. Taking Theorem \ref{tor} and Theorem \ref{vital} together, we know that when $\cF_4(Q) \neq \emptyset$, almost all fields $K \in \cF_4(Q;X)$ satisfy (\ref{toreq}).
We might expect that $|\cF_4(Q;X)| \sim CX^{1/2}$ for some positive constant $C$. But to show that Theorem \ref{tor} and our third main result -- an effective Chebotarev density theorem (see Theorem \ref{main}) -- hold for almost all fields in $\cF_4(Q;X)$, it suffices to find any constant $\beta>0$ such that $|\cF_4(Q;X)| \gg_{Q} X^{\beta}$. 

We say that $b$ is a norm of $\Q(\sqrt{a})$ if $b$ is a norm of an element in $\Q(\sqrt{a})$. For the set $\cF_4(Q)$ to be nonempty, consider the following three criteria on $a,b$:



\begin{equation}\label{co1}
b \text{ is a norm of } \Q(\sqrt{a});
\end{equation} 
\begin{equation}\label{co2}
-b \text{ is a norm of } \Q(\sqrt{a});
\end{equation}
\begin{equation}\label{co3}
-a \text{ is a norm of } \Q(\sqrt{b}). 
\end{equation}

We will prove that $\cF_4(Q) \neq \emptyset$ is equivalent to the following condition:

\begin{cond}\label{1234}
The pair $(a,b)$ satisfies at least one of (\ref{co1}), (\ref{co2}), or (\ref{co3}).



\end{cond}

It is easy to see that the relationship of $a$ and $b$ is independent of the order, i.e., if the pair $(a,b)$ satisfies Condition \ref{1234}, then the pair $(b,a)$ does as well.
Note that we also have symmetry in Condition \ref{1234} among $a,b,\frac{ab}{\xi^2}$ (recall that $\xi=\mathrm{gcd}(|a|,|b|)$).
By this we mean that if the pair $(a,b)$ satisfies Condition \ref{1234}, then so does $(a,\frac{ab}{\xi^2})$, or $(b,\frac{ab}{\xi^2})$, and vice versa. This is because $b$ is a norm of $\Q(\sqrt{a}) \iff a$ is a norm of $\Q(\sqrt{b})$; $\frac{ab}{\xi^2}$ is a norm of $\Q(\sqrt{a}) \iff -b$ is a norm of $\Q(\sqrt{a})$ (since $-a=(-\sqrt{a})\cdot (\sqrt{a})$ is a norm of $\Q(\sqrt{a})$).

Moreover, there are infinitely many pairs $(a,b)$ satisfying Condition \ref{1234}. For example, when $b_1$ is a prime and $b_1 \equiv \pm 1 \ (\mathrm{mod} \ 8)$, $(2,b_1)$ satisfies (\ref{co1}); when $b_2$ is a prime and $b_2 \equiv 11 \ (\mathrm{mod} \ 12)$, $(3,b_2)$ satisfies (\ref{co2}). There are infinitely many such $b_1,b_2$, by Dirichlet's theorem on primes in arithmetic progressions.

\begin{rem}\label{dif}
\normalfont
One notices that we only deal with infinite families of $D_4$-quartic fields but not all $D_4$-quartic fields. If we take the union of all the $\le B_1X^{\ep}$ exceptional fields as $a,b$ vary, we possibly get $\gg X$ exceptional fields. It remains an interesting open problem to prove the analogue of Theorem \ref{tor} for all $D_4$-quartic fields simultaneously. 
\end{rem}

\subsection*{Outline of the method}

At its foundation, our approach is analogous to that of \cite{PTW17}. The difference from \cite{PTW17} will be shown explicitly in Section \ref{mot}. After the work of Ellenberg and Venkatesh in \cite{EV07}, to prove (\ref{toreq}) in Theorem \ref{tor} for a number field $K$, it will suffice to be able to count the number of small unramified primes which split completely in $K$ (see Proposition \ref{ev}). Our main idea is that after fixing a biquadratic field $Q=\Q(\sqrt{a},\sqrt{b})$, we establish a new effective Chebotarev density theorem (Theorem \ref{main}) for almost all fields in $\cF_4(Q;X)$. In particular, studying the family $\cF_4(Q)$, as was recommended in Remark 6.12 of \cite{PTW17}, avoids the barrier encountered in \cite{PTW17} when considering $D_4$-quartic fields; see Section \ref{mot}. 

For a $D_4$-quartic field $K$ and its Galois closure $\widetilde{K}$, and for any fixed conjugacy class $\cC$ in $G \cong D_4$, we define the prime counting function as 
\begin{equation*}
\pi_\cC(x,\widetilde{K}/\Q)=| \{p \text{ prime}: \ p \text{ is unramified in }\widetilde{K}, \left[ \frac{\widetilde{K}/\Q}{p} \right]=\cC, p \le x \} |,
\end{equation*}
where $\left[ \frac{\widetilde{K}/\Q}{p} \right]$ is the Artin symbol, i.e., the conjugacy class of the Frobenius element corresponding to the extension $\widetilde{K}/\Q$ and the prime $p$. 

Obtaining an accurate count for $\pi_\cC(x,\widetilde{K}/\Q)$ depends on a zero-free region for $\zeta_{\widetilde{K}}(s)$. Thus, we consider the factorization of $\zeta_{\widetilde{K}}(s)$, i.e.,
\begin{equation}\label{l}
\zeta_{\widetilde{K}}(s)=\ \prod_{\mathclap{\rho \in \mathrm{Irr}(D_4)}} \ L(s,\rho,\widetilde{K}/\Q)^{\dim \rho}=\zeta(s)L(s,\chi_{a^\ast})L(s,\chi_{b^\ast})L(s,\chi_{({\frac{ab}{\xi^2}})^\ast})L^2(s,\rho_{\widetilde{K}}).
\end{equation}
Here we use the notation that for $c \in \{a,b,\frac{ab}{\xi^2} \}$, 
\begin{equation*}
c^\ast=
\begin{cases}
c, & \text{ if } c \equiv 1 \ (\mathrm{ mod } \ 4) \\
4c, & \text{ if } c \equiv 2,3 \ (\mathrm{ mod } \ 4)
\end{cases}
\end{equation*}
is the fundamental discriminant of the field $\Q(\sqrt{c})$ over $\Q$, and $\chi_{c^\ast}(\cdot)=\left( \frac{c^{\ast}}{\cdot} \right)$ is the real primitive Dirichlet character given by the Kronecker symbol. Note also that $\rho_{\widetilde{K}}$ is the 2-dimensional faithful representation of $D_4$. 

Since we have fixed $a$ and $b$ as $K$ varies, the $L$-functions $L(s,\chi_{a^\ast})$, $L(s,\chi_{b^\ast})$, and $L(s,\chi_{({\frac{ab}{\xi^2}})^\ast})$ are fixed in (\ref{l}), and hence so is the Dedekind zeta function of the biquadratic field,
\begin{equation*}
\zeta_Q(s)=\zeta(s)L(s,\chi_{a^\ast})L(s,\chi_{b^\ast})L(s,\chi_{({\frac{ab}{\xi^2}})^\ast}).
\end{equation*}
Therefore, as $K$ varies in $\cF_4(a,b)$, the only varying $L$-factor in $\zeta_{\widetilde{K}}(s)$ is $L(s,\rho_{\widetilde{K}})$. This is critical to the success of our method, see Remark \ref{suc}.

We first prove a Chebotarev density theorem with an assumed zero-free region for $\zeta_{\widetilde{K}}(s)/\zeta_Q(s)$.

\begin{thm}[Chebotarev density theorem with assumed zero-free region]\label{cdtzf}
Let $0<\ep_0<\frac14$ be sufficiently small. Suppose that $Q=\Q(\sqrt{a},\sqrt{b})$ is a biquadratic field with $\cF_4(Q) \neq \emptyset$. Suppose also that for $K \in \cF_4(Q)$ such that $D_{\widetilde{K}} \ge C_7$ for an absolute constant $C_7$ given in (\ref{c7}), $\zeta_{\widetilde{K}}(s)/\zeta_Q(s)=L^2(s,\rho_{\widetilde{K}})$ (hence $L(s,\rho_{\widetilde{K}})$) has no zero in \begin{equation}\label{azfr2}
[1-\delta,1] \times [-(\log D_{\widetilde{K}})^{2/\delta},(\log D_{\widetilde{K}})^{2/\delta}],
\end{equation}
where 
\begin{equation}\label{delta}
\delta=\frac{\ep_0}{42+4\ep_0}.
\end{equation} 
Then for every conjugacy class $\cC \subset G=D_4$,
\begin{equation}\label{bnd}
\left| \pi_{\cC}(x,\widetilde{K}/\Q)-\frac{|\cC|}{|G|}\mathrm{Li}(x) \right| \le \frac{|\cC|}{|G|} \frac{x}{(\log x)^2}
\end{equation}
for all
\begin{equation*}
x \ge \kappa_1 \exp{[\kappa_2(\log \log(D_{\widetilde{K}}^{\kappa_3}))^{2}]} 
\end{equation*}
for parameters $\kappa_i=\kappa_i(a,b,\ep_0)$ (see (\ref{kappa1}), (\ref{kappa2}), (\ref{kappa3})).
\end{thm}

Theorem \ref{cdtzf} is analogous to Theorem 3.1 in \cite{PTW17} and we will prove Theorem \ref{cdtzf} mainly by an adaptation of the proof of Theorem 3.1 in \cite{PTW17}.

We show via work of Kowalski and Michel in \cite{KM02} that almost all fields in our family are zero-free in the described region.

\begin{thm}\label{qwe}
Suppose that $Q=\Q(\sqrt{a},\sqrt{b})$ is such that $\cF_4(Q) \neq \emptyset$. For every $0<\ep_0<\frac14$, there are $\ll_{\ep_0} X^{\ep_0}$ fields $K \in \cF_4(Q;X)$ such that $\zeta_{\widetilde{K}}(s)/\zeta_Q(s)=L^2(s,\rho_{\widetilde{K}})$ could have a zero in the region (\ref{azfr2}).
\end{thm}
 
Hence we obtain our third main result, an effective Chebotarev density theorem for our family $\cF_4(Q)$.
 
\begin{thm}\label{main}
Suppose that $Q=\Q(\sqrt{a},\sqrt{b})$ is such that $\cF_4(Q) \neq \emptyset$. For every $0<\ep_0 < \frac14$ sufficiently small, there exists a constant $B_2=B_2(\ep_0)$ such that for every $X \ge 1$, aside from at most $B_2X^{\ep_0}$ fields in $\cF_4(Q;X)$, each field $K \in \cF_4(Q;X)$ has the property that for every conjugacy class $\cC \subset G=D_4$,
\begin{equation}\label{bnd}
\left| \pi_{\cC}(x,\widetilde{K}/\Q)-\frac{|\cC|}{|G|}\mathrm{Li}(x) \right| \le \frac{|\cC|}{|G|} \frac{x}{(\log x)^2}
\end{equation}
for all
\begin{equation}\label{bnd2}
x \ge \kappa_1 \exp{[\kappa_2(\log \log(D_{\widetilde{K}}^{\kappa_3}))^{2}]} 
\end{equation}
for parameters $\kappa_i=\kappa_i(a,b,\ep_0)$.
\end{thm}

Theorem \ref{main} is a direct consequence of Theorem \ref{cdtzf} and Theorem \ref{qwe}. The error term in the Chebotarev density theorem (Theorem \ref{main}) can be improved by considerations of Brumley, Thorner, and Zaman, see \cite{TZ18}, but as this is not needed for the application of Theorem \ref{tor}, we do not pursue this here.

Taking Theorem \ref{vital} and Theorem \ref{main} together, we know that almost all fields $K \in \cF_4(Q;X)$ have the property that for every conjugacy class $\cC \subset G=D_4$, (\ref{bnd}) and (\ref{bnd2}) hold.

\end{section}

\begin{section}{Motivation for the construction of the family $\cF_4(Q)$}\label{mot}

We begin by describing the family $\cF_4(Q)$ and in particular why in this setting we can carry out the approach of \cite{PTW17}.

For a $D_4$-quartic field $K$, we consider all the subextensions of $\widetilde{K}$. In order to understand the relations among $K$, $\widetilde{K}$, and their subfields, we are led by Galois theory to consider all the subgroups of $D_4$. We write $D_4=\langle r,s \ | \ r^4=1, s^2=1, srs^{-1}=r^{-1} \rangle$ and then have the following diagrams.

\noindent Lattice of groups:
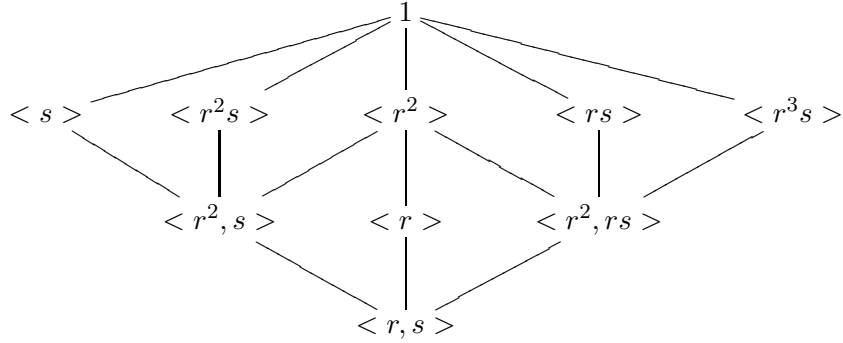
\[
\xymatrix{
 & & 1 \ar@{-}[lld]  \ar@{-}[ld] \ar@{-}[d] \ar@{-}[dr] \ar@{-}[drr] & \\
 <s> \ar@{-}[dr] & <r^2s> \ar@{-}[d] & <r^2> \ar@{-}[ld] \ar@{-}[d] \ar@{-}[dr] & <rs> \ar@{-}[d] & <r^3s> \ar@{-}[ld] \\
 & <r^2,s> \ar@{-}[dr] & <r> \ar@{-}[d] & <r^2,rs> \ar@{-}[ld] \\
 & & <r,s> & & }
\]
Lattice of fields:
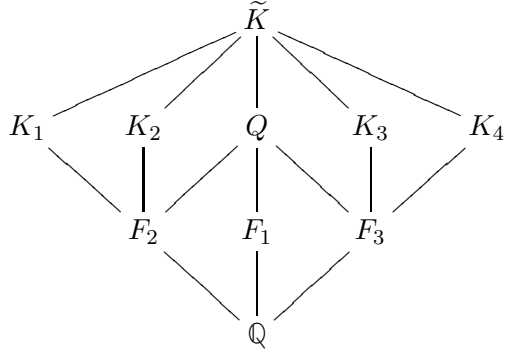
\[
\xymatrix{
 & & \widetilde{K} \ar@{-}[lld] \ar@{-}[ld] \ar@{-}[d] \ar@{-}[dr] \ar@{-}[drr] & & \\
K_1 \ar@{-}[dr] & K_2 \ar@{-}[d] & Q \ar@{-}[ld] \ar@{-}[d] \ar@{-}[dr] & K_3  \ar@{-}[d] & K_4 \ar@{-}[ld] \\
 & F_2 \ar@{-}[dr] & F_1 \ar@{-}[d] & F_3 \ar@{-}[ld] \\
 & & \mathbb{Q} & & }
\]
In the lattice of fields, the set of fields $\{F_1, F_2, F_3\}$ is equal to the set $\{\mathbb{Q}(\sqrt{a}), \mathbb{Q}(\sqrt{b}), \mathbb{Q}(\frac{\sqrt{ab}}{\xi})\}$, where $\xi=\mathrm{gcd}(|a|,|b|)$.

At a key step in \cite{PTW17}, for each fixed group $G$, the authors provide a way to control the number of $G$-fields whose Galois closures share a certain fixed field. In detail, by specifying an appropriate restriction on the ramification type of tamely ramified primes, one can impose that if the primes divide $D_K$, then they divide $D_F$. Here $F=\widetilde{K}^H$, where $H$ is allowed to be the kernel of any irreducible representation of the Galois group $G$. In the case $G=D_4$, this cannot be done, since there is no restriction on ramification type satisfying the requirement above. We illustrate this point with the following table (recall that $D_4=\langle r,s \ | \ r^4=1, s^2=1, srs^{-1}=r^{-1} \rangle$). 

In this table, $p$ is an odd prime. This makes $p$ unramified or tamely ramified since $p \nmid |D_4|$ (see Lemma 6.10 of \cite{PTW17}). Hence, the inertia group of $p$ is cyclic. The first column is the conjugacy class of a generator for the cyclic inertia group of $p$. In the first row, $\exp_p(D_K)$ denotes the exponent $\alpha$ such that $p^\alpha || D_K$, and $F_1,F_2,F_3$ are the same as depicted in the lattice of fields. Note that all of the fields $F_1,F_2,F_3$ are of the form $K^H$, where $H$ varies over the kernels of the irreducible representations of the Galois group $G$.

\begin{table}[!ht]\label{table2.3.1}
\begin{tabular}{|p{39mm}|p{17mm}|p{17mm}|p{17mm}|p{17mm}|p{17mm}|}
\hline
$\text{Ramification type of } p$ & $\exp_p(D_K)$ & $\exp_p(D_{\widetilde{K}})$ & $\exp_p(D_{F_1})$ & $\exp_p(D_{F_2})$ & $\exp_p(D_{F_3})$ \\
\hline
$[1]$ & $0$ & $0$ & $0$ & $0$ & $0$ \\
\hline
$[r^2]$ & 2 & 4 & 0 & 0 & 0 \\
\hline
$[s]$ & 1 & 4 & 1 & 0 & 1 \\ 
\hline
$[r^3s]$ & 2 & 4 & 1 & 1 & 0 \\
\hline
$[r]$ & 3 & 6 & 0 & 1 & 1 \\
\hline
\end{tabular}
\end{table}

From the table we know that whatever ramification type (or collection of ramification types) we choose, there are primes $p$ such that $p \mid D_K$ but $p \nmid D_F$ for a field $F=K^H$. Therefore, there is no suitable restriction on ramification type for the method presented in \cite{PTW17}. This motivates our definition of the family $\cF_4(Q)$, which effectively removes consideration of the last three columns of this table. 

\begin{rem}\label{suc}
\normalfont
Once the biquadratic field $Q=\Q(\sqrt{a},\sqrt{b})$ has been fixed, the $L$-functions $L(s,\chi_{a^\ast})$, $L(s,\chi_{b^\ast})$, and $L(s,\chi_{({\frac{ab}{\xi^2}})^\ast})$ are fixed. Our method is to pass to the right of possible real simple exceptional zeros of these three $L$-functions. When $K$ varies in $\cF_4(Q;X)$, we obtain a zero-free region (\ref{azfr2}) for almost all functions $L(s,\rho_{\widetilde{K}})$. Then we consider the intersection of the zero-free regions of $L(s,\rho_{\widetilde{K}})$ and $L(s,\chi_{a^\ast})L(s,\chi_{b^\ast})L(s,\chi_{({\frac{ab}{\xi^2}})^\ast})$. We are able to obtain a zero-free region of $\zeta_{\widetilde{K}}(s)/\zeta(s)$ (see (\ref{7z})) for almost all fields $K \in \cF_4(Q;X)$. Based on the zero-free region of $\zeta_{\widetilde{K}}(s)/\zeta(s)$, we obtain an effective Chebotarev density theorem (Theorem \ref{main}) and a theorem on $\ell$-torsion of class groups (Theorem \ref{tor}).
\end{rem}

We also remark that the paper \cite{EPW17} gives a nontrivial bound for $\ell$-torsion in class groups of non-$D_4$ quartic fields. The obstacle in the $D_4$ case lies in the counting problem for $D_4$-quartic fields with local conditions, which has non-multiplicative local densities; on the other hand, see (2.14), (2.15) of \cite{EPW17} on the counting for non-$D_4$ quartic fields.

\end{section}

\begin{section}{Counting $D_4$-quartic fields with a fixed biquadratic field $Q$}\label{core}

In this section, the problem that interests us is a lower bound of $|\cF_4(Q;X)|$ as $X \to \infty$, provided that $\cF_4(Q) \neq \emptyset$. The aim of this section is to describe a new explicit construction for this problem. 

We first state all the necessary lemmas and propositions, and then turn to the proofs.
As before, $Q=\Q(\sqrt{a},\sqrt{b})$ is a biquadratic field, where $a,b$ are distinct square-free integers not equal to $0$ or $1$.

\begin{lem}\label{a}
For $K \in \cF_4(Q)$, there is a unique quadratic subfield of $K$. Moreover, $\widetilde{K}$ only contains $D_4$-quartic fields that are extensions of two quadratic subfields of $\widetilde{K}$, but not of the third quadratic subfield.
\end{lem}

Lemma \ref{a} is a direct corollary of the lattice of fields, whose notation we now adopt. If we consider the $D_4$-quartic field $K_1$, then its unique quadratic subfield is $F_2$. Moreover, $\widetilde{K_1}$ only contains $D_4$-quartic fields that are extensions of $F_2$ and $F_3$, but not of $F_1$. In this case, we say that $F_2$ and $F_3$ are the only two \emph{extended quadratic} fields of $K_1$.

With Lemma \ref{a} in hand, we define 
\begin{equation}
\cF_4(a,b)=\{ K \in \cF_4(Q): \Q(\sqrt{a}), \Q(\sqrt{b}) \text{ are the only two extended quadratic fields of } K \}.
\end{equation} 
Similarly, we define $\cF_4(a,\frac{ab}{\xi^2})$ and $\cF_4(b,\frac{ab}{\xi^2})$, where $\xi=\mathrm{gcd}(|a|,|b|)$. Then we have
\begin{equation*}
\cF_4(Q)=\cF_4(a,b) \sqcup \cF_4(a,\frac{ab}{\xi^2}) \sqcup \cF_4(b,\frac{ab}{\xi^2})
\end{equation*}
since for any field in $\cF_4(Q)$, its extended quadratic fields are $\Q(\sqrt{a})$ and $\Q(\sqrt{b})$, $\Q(\sqrt{a})$ and $\Q(\frac{\sqrt{ab}}{\xi})$, or $\Q(\sqrt{b})$ and $\Q(\frac{\sqrt{ab}}{\xi})$.

For $c \in \{a,b\}$, we are able to define a subset $\cF_{4,c}(a,b)$ of $\cF_4(a,b)$ by
\begin{equation*}
\cF_{4,c}(a,b)=\{K \in \cF_4(a,b): \Q(\sqrt{c}) \text{ is the unique quadratic subfield of } K\}
\end{equation*}
and a subset $\cF_{4,c}(a,b;X)$ of $\cF_{4,c}(a,b)$ by
\begin{equation*}
\cF_{4,c}(a,b;X)=\cF_4(Q;X) \cap \cF_{4,c}(a,b).
\end{equation*}
It is clear that
\begin{equation*}
\cF_{4}(a,b)=\cF_{4,a}(a,b) \sqcup \cF_{4,b}(a,b).
\end{equation*}
Similarly as in our definitions of two refined subfamilies inside $\cF_4(a,b)$, we have definitions of two refined subfamilies inside each of $\cF_4(a,\frac{ab}{\xi^2})$ and $\cF_4(b,\frac{ab}{\xi^2})$. Hence, we have
\begin{equation}\label{six}
\cF_4(Q)=\cF_{4,a}(a,b) \sqcup \cF_{4,b}(a,b) \sqcup \cF_{4,a}(a,\frac{ab}{\xi^2}) \sqcup \cF_{4,\frac{ab}{\xi^2}}(a,\frac{ab}{\xi^2}) \sqcup \cF_{4,b}(b,\frac{ab}{\xi^2}) \sqcup \cF_{4,\frac{ab}{\xi^2}}(b,\frac{ab}{\xi^2}).
\end{equation}
Thus in order to give a lower bound on $|\cF_4(Q;X)|$, it will suffice to give a lower bound on one of these six subfamilies. From now on, we focus on $\cF_{4,a}(a,b)$, but the results apply to the other five subfamilies as well.

We have the following result that gives a generator for $K \in \cF_{4,a}(a,b)$.

\begin{lem}\label{cda}
For $K \in \cF_{4,a}(a,b)$, there exists $g \in \Z$, $h \in \Z^*=\Z-\{0\}$ such that $K=\Q(\sqrt{g+h\sqrt{a}})$.


\end{lem}

Assuming Lemma \ref{cda}, we give explicit criteria for $\cF_{4,a}(a,b)$ to be nonempty. 

\begin{prop}\label{cond}
Assume that $Q=\Q(\sqrt{a},\sqrt{b})$ is a biquadratic extension of $\Q$. Then $\cF_{4,a}(a,b) \neq \emptyset$ if and only if $b$ is a norm of $\Q(\sqrt{a})$.










Moreover, under the condition $\cF_{4}(Q) \neq \emptyset$, there exists a well-defined function $\varphi: (a,b) \mapsto (g_0,h_0,n_0)$, the image being an ordered triple of positive integers satisfying a certain equation (see (\ref{ghnn}) if $\cF_{4,a}(a,b) \neq \emptyset$ and 
(\ref{sub1}) (\ref{sub2}) if $\cF_{4,a}(a,b)=\emptyset$). The triple depends only on the ordered pair $(a,b)$. 
\end{prop}

By Proposition \ref{cond}, we have 
$\cF_{4,a}(a,b) \neq \emptyset \iff b$ is a norm of $\Q(\sqrt{a}) \iff a$ is a norm of $\Q(\sqrt{b}) \iff \cF_{4,b}(a,b) \neq \emptyset$. Similarly, we have $\cF_{4,a}(a,\frac{ab}{\xi^2}) \neq \emptyset \iff \cF_{4,\frac{ab}{\xi^2}}(a,\frac{ab}{\xi^2}) \neq \emptyset$; $\cF_{4,b}(b,\frac{ab}{\xi^2}) \neq \emptyset \iff \cF_{4,\frac{ab}{\xi^2}}(b,\frac{ab}{\xi^2}) \neq \emptyset$.
Proposition \ref{cond} then gives that $\cF_4(Q) \neq \emptyset$ is equivalent to the statement that $(a,b)$ satisfies Condition \ref{1234}, because both are equivalent to the statement that at least one of the sets $\cF_{4,a}(a,b), \cF_{4,a}(a,\frac{ab}{\xi^2}), \cF_{4,b}(b,\frac{ab}{\xi^2})$ is nonempty.

For the moment, we assume that $\cF_{4,a}(a,b) \neq \emptyset$ and show that $|\cF_4(Q;X)| \gg_Q X^{1/2}$. Later we show how to reach the same conclusion, if $\cF_{4,a}(a,b)=\emptyset$ and at least one of the other two subfamilies is nonempty.

\begin{lem}\label{ab} 
Using Lemma \ref{cda} and Proposition \ref{cond}, we let $K=\Q(\sqrt{g+h\sqrt{a}}) \in \cF_{4,a}(a,b)$, where $g,h$ are integers, $h \neq 0$, and let a positive integer $n$ be given such that $n$ satisfies 
\begin{equation*}
g^2-h^2a=n^2b.
\end{equation*}
We will show that such an integer $n$ exists and is unique; see (\ref{n2b}) 
below. Then we have
\begin{equation}\label{dk}
D_K \le C_{a,b}n^2,
\end{equation}
where $C_{a,b}=256|a|^3|b|^3$. Moreover, under the analogous conditions, the same result (\ref{dk}) holds for fields in the other five subfamilies in (\ref{six}) with the same constant $C_{a,b}$.
\end{lem}

Given an ordered pair $(a,b)$ with $\cF_4(\Q(\sqrt{a},\sqrt{b})) \neq \emptyset$, we recall the function $\varphi:(a,b) \mapsto (g_0,h_0,n_0)$ in Proposition \ref{cond} and set
\begin{equation}\label{defp}
M_a(a,b;X)=\{m \in \Z_{>0} \ \text{square-free} \mid \ \mathrm{gcd}(m,|ab|)=1, \ m \le \frac{1}{16n_0\sqrt{|a|^3|b|^3}} X^{1/2} \}. 
\end{equation}
We also set
\begin{equation*}
K_{[m]}=\Q(\sqrt{g_0m+h_0m\sqrt{a}})
\end{equation*}
for any positive integer $m$ and
\begin{equation*}
T_a(a,b;X)=\{K_{[m]}: \ m \in M_a(a,b;X) \}.
\end{equation*}

Assuming $\cF_{4,a}(a,b) \neq \emptyset$, we have the following lower bound on $|T_a(a,b;X)|$. This gives a lower bound of $|\cF_{4,a}(a,b;X)|$.

\begin{prop}\label{comb}
Assume that $\cF_{4,a}(a,b) \neq \emptyset$.
Then the following statements hold.

(1) We have
\begin{equation}\label{ta}
T_a(a,b;X) \subset \cF_{4,a}(a,b;X).
\end{equation}

(2) If $m_1,m_2 \in M_a(a,b;X)$, $m_1 \neq m_2$, then $K_{[m_1]} \neq K_{[m_2]}$.

(3) We have
\begin{equation}\label{tb}
|T_a(a,b;X)| \gg_{a,b} X^{1/2}.
\end{equation}
\end{prop}

Assuming the above lemmas and propositions, we deduce Theorem \ref{vital} as follows. Let $Q=\Q(\sqrt{a},\sqrt{b})$ be such that $\cF_4(Q) \neq \emptyset$. If $\cF_{4,a}(a,b) \neq \emptyset$, then (\ref{ta}) and (\ref{tb}) immediately give Theorem \ref{vital}. If $\cF_{4,a}(a,b)=\emptyset$, then $\cF_{4,a}(a,\frac{ab}{\xi^2}) \neq \emptyset$ or $\cF_{4,b}(b,\frac{ab}{\xi^2}) \neq \emptyset$. We choose any $c \in \{a,b\}$ such that $\cF_{4,c}(c,\frac{ab}{\xi^2}) \neq \emptyset$ holds and denote $\widehat{a}=c$, $\widehat{b}=\frac{ab}{\xi^2}$. Thus $\cF_{4,\widehat{a}}(\widehat{a},\widehat{b}) \neq \emptyset$. Note that $\Q(\sqrt{\widehat{a}},\sqrt{\widehat{b}})=\Q(\sqrt{a},\sqrt{b})$
and that $(\widehat{a},\widehat{b})$ satisfies (\ref{co1}). Again by Proposition \ref{cond}, the image of $(\widehat{a},\widehat{b})$ under $\varphi$ is a triple of integers.
Proposition \ref{comb} holds with $a$ replaced by $\widehat{a}$, $b$ replaced by $\widehat{b}$. Then (\ref{ta}) and (\ref{tb}) give Theorem \ref{vital}. 

\vspace{5mm}
\noindent\textbf{Proof of the lemmas and propositions}
\begin{proof}[Proof of Lemma \ref{cda}] 
Noting that $K$ is a degree 2 extension of $\Q(\sqrt{a})$, we can find $\alpha \in K \backslash \Q(\sqrt{a})$ s.t. $\alpha^2 \in \Q(\sqrt{a})$. Clearly $K=\Q(\sqrt{a})(\alpha)$. Every element in $\Q(\sqrt{a})$ has the form $u+v\sqrt{a}$, where $u,v \in \Q$, so $\alpha^2=g'+h'\sqrt{a}$, where $g',h' \in \Q$. 
If $h'=0$, then $\Q(\alpha)$ is a normal extension of $\Q$, leading to a contradiction since $K \neq \widetilde{K}$. Thus we have $h' \in \Q^{*}$. 
Letting $\lambda>0$ be the least common multiple of the denominators of $g'$ and $h'$ (if $g'=0$, then 1 is its denominator), we have $\lambda^2\alpha^2=\lambda^2g'+\lambda^2h'\sqrt{a}$. Now we let $g=\lambda^2g', h=\lambda^2h'$; then $g \in \Z$, $h \in \Z^*$, and $\sqrt{g+h\sqrt{a}} \in K$. 
Since $\alpha \notin \Q(\sqrt{a})$, we have $\sqrt{g+h\sqrt{a}} \notin \Q(\sqrt{a})$. Thus, $\Q(\sqrt{g+h\sqrt{a}})$ is a quartic subextension of $K$. It forces $K=\Q(\sqrt{g+h\sqrt{a}})$.
\end{proof}

\begin{proof}[Proof of Proposition \ref{cond}]

For the first direction, we assume $\cF_{4,a}(a,b) \neq \emptyset$. By Lemma \ref{cda}, for $K \in \cF_{4,a}(a,b)$, we can write $K=\Q(\sqrt{g+h\sqrt{a}})$ where $g \in \Z$, $h \in \Z^*$. It is easy to see that 
\begin{equation*}
\widetilde{K}=K(\sqrt{g-h\sqrt{a}})=\Q(\sqrt{g+h\sqrt{a}},\sqrt{g-h\sqrt{a}}).
\end{equation*}
We notice that $\sqrt{g+h\sqrt{a}} \cdot \sqrt{g-h\sqrt{a}}=\sqrt{g^2-h^2a} \in \Q(\sqrt{a}),\Q(\sqrt{b}) \text{ or } \Q(\frac{\sqrt{ab}}{\xi})$. It cannot be in $\Q(\sqrt{a})$, otherwise $\sqrt{g-h\sqrt{a}}=(\sqrt{g+h\sqrt{a}})^{-1}\sqrt{g^2-h^2a} \in K$, which leads to a contradiction since $K \neq \widetilde{K}$. Therefore we have $\sqrt{g^2-h^2a} \in \Q(\sqrt{b}) \cup \Q(\frac{\sqrt{ab}}{\xi})$, which induces that either 
\begin{equation}\label{n2b}
g^2-h^2a=n^2b
\end{equation} or 
\begin{equation}\label{n2ab}
g^2-h^2a=n^2\frac{ab}{\xi^2},
\end{equation}
for some $n \in \Z_{>0}$.

To study the automorphisms of $\widetilde{K}/\Q$, we use the notations $\mathbb{1}=\sqrt{g+h\sqrt{a}}$, $\mathbb{2}=\sqrt{g-h\sqrt{a}}$, $\mathbb{3}=-\sqrt{g+h\sqrt{a}}$, and $\mathbb{4}=-\sqrt{g-h\sqrt{a}}$. Then we see $\mathrm{Gal}(\widetilde{K}/\Q) \cong D_4=\langle r,s \ | \ r^4=1, s^2=1, srs^{-1}=r^{-1} \rangle$ by setting $r=(\mathbb{1}\mathbb{2}\mathbb{3}\mathbb{4})$ and $s=(\mathbb{1}\mathbb{3})$ as permutations of four letters. The field $K=\Q(\sqrt{g+h\sqrt{a}})$ appears as the fixed field of $\langle r^2s \rangle$ (i.e., $K_2$ in the lattice of fields) and the field $\Q(\sqrt{g-h\sqrt{a}})$ appears as the fixed field of $\langle s \rangle$ (i.e., $K_1$ in the lattice of fields). Now we would like to compute the fixed field of $\langle rs \rangle$. Using the temporary notation $A=\sqrt{g+h\sqrt{a}}$ and $B=\sqrt{g-h\sqrt{a}}$, we notice that $A-B$ is invariant under the automorphism $rs=(\mathbb{1}\mathbb{4})(\mathbb{2}\mathbb{3})$. Moreover, $A^2-B^2=2h\sqrt{a}$ and $A^2+B^2=2g$. Thus,
\begin{equation*}
(A-B)^2=2g-2AB=2g-2\sqrt{g^2-h^2a}.
\end{equation*}

If we have (\ref{n2b}), then $(A-B)^2=2g-2n\sqrt{b}$ and $A-B=\pm \sqrt{2g-2n\sqrt{b}}$. The fixed field of $\langle rs \rangle$ is $\Q(\sqrt{2g-2n\sqrt{b}})$. Thus, the extended quadratic fields of $K$ are $\Q(\sqrt{a})$ and $\Q(\sqrt{b})$.

If we have (\ref{n2ab}), then $(A-B)^2=2g-2n\xi^{-1}\sqrt{ab}$ and $A-B=\pm \sqrt{2g-2n\xi^{-1}\sqrt{ab}}$. The fixed field of $\langle rs \rangle$ is $\Q(\sqrt{2g-2n\xi^{-1}\sqrt{ab}})$. Thus, the extended quadratic fields of $K$ are $\Q(\sqrt{a})$ and $\Q(\sqrt{\frac{ab}{\xi^2}})$. This case is excluded because $K \in \cF_{4,a}(a,b)$. Thus, we must have (\ref{n2b}), which is equivalent to the fact that $b$ is a norm of $\Q(\sqrt{a})$.










For the other direction 
of Proposition \ref{cond}, 
we know that (\ref{n2b}) has a nontrivial solution because $b$ is a norm of $\Q(\sqrt{a})$. 
We fix a well-defined nontrivial solution $(g_0,h_0,n_0)$ of (\ref{n2b}) 
in the following way. 
First, we define the set $N=\{n \in \Z_{>0}: \ \exists g \in \Z_{\ge0}, h \in \Z_{>0} \text{ s.t. } g^2-h^2a-n^2b=0\}$.
The set $N$ is nonempty and there exists the least element $n_0$ in $N$. Second, we define the set $H=\{h \in \Z_{>0}: \ \exists g \in \Z_{\ge0} \text{ s.t. } g^2-h^2a-n_0^2b=0\}$. 
There exists the least element $h_0$. Once $n_0,h_0$ are determined, $g_0$ is uniquely determined as the nonnegative solution of $g^2-h_0^2a-n_0^2b=0$. 
In this way can we choose the well-defined solution $(g_0,h_0,n_0)$, i.e., $g_0,h_0,n_0$ are nonnegative integers only depending on the ordered pair $(a,b)$ and satisfy $h_0 \neq 0$, $n_0 \neq 0$,
\begin{equation}\label{ghnn}
g_0^2-h_0^2a=n_0^2b. \\
\end{equation}

As an example, if $(a,b)=(2,7)$, then $(g_0,h_0,n_0)=(3,1,1)$.

We denote $K_0=\Q(\sqrt{g_0+h_0\sqrt{a}})$ and claim that $K_0 \in \cF_{4,a}(a,b)$. We see that $K_0$ is quartic and has a quadratic subfield $\Q(\sqrt{a})$. By \cite[p.88]{HSW91}, we know that $\Gal(\widetilde{K_0}/\Q) \cong K_4 \iff g_0^2-h_0^2a=k^2$ for some integer $k$; $\Gal(\widetilde{K_0}/\Q) \cong C_4 \iff g_0^2-h_0^2a=ak^2$ for some integer $k$; $\Gal(\widetilde{K_0}/\Q) \cong D_4 \iff g_0^2-h_0^2a \neq k^2 \text{ or } ak^2$ for any integer $k$. Here $K_4$ is the Klein 4-group and $C_4$ is the cyclic group of order 4.
Since $g_0^2-h_0^2a=n_0^2b$ and $n_0^2b \neq k^2 \text{ or } ak^2$ for any $k$, we have $\Gal(\widetilde{K_0}/\Q) \cong D_4$. Moreover, $\widetilde{K_0}=\Q(\sqrt{g_0+h_0\sqrt{a}})(\sqrt{g_0-h_0\sqrt{a}})$ contains $\sqrt{g_0^2-h_0^2a}$, an element of $\Q(\sqrt{a},\sqrt{b}) \backslash \Q(\sqrt{a})$. Thus, $\widetilde{K_0}$ contains $\Q(\sqrt{a},\sqrt{b})$ and hence $K_0 \in \cF_4(Q)$. Following the proof of the first direction of Proposition \ref{cond}, we know that $\Q(\sqrt{2g_0-2n_0\sqrt{b}})$ is in the same lattice of fields as $K_0$, so the extended quadratic fields of $K_0$ are $\Q(\sqrt{a})$ and $\Q(\sqrt{b})$. This gives $K_0 \in \cF_4(a,b)$. Recalling that $K_0$ has a quadratic subfield $\Q(\sqrt{a})$, we conclude that $K_0 \in \cF_{4,a}(a,b)$ and $\cF_{4,a}(a,b) \neq \emptyset$.


The function $\varphi$ can be constructed as follows. We assume that $\cF_4(Q) \neq \emptyset$. If $\cF_{4,a}(a,b) \neq \emptyset$, then we let $\varphi$ send $(a,b)$ to $(g_0,h_0,n_0)$ as above. If $\cF_{4,a}(a,b)=\emptyset$, then $\cF_{4,a}(a,\frac{ab}{\xi^2}) \neq \emptyset$ or $\cF_{4,b}(b,\frac{ab}{\xi^2}) \neq\emptyset$. 
If $\cF_{4,a}(a,\frac{ab}{\xi^2}) \neq \emptyset$, then we choose $(g_0,h_0,n_0)$ to be the triple satisfying 
\begin{equation}\label{sub1}
g_0^2-h_0^2a=n_0^2\frac{ab}{\xi^2}
\end{equation}
using the same procedure as above.
If $\cF_{4,a}(a,\frac{ab}{\xi^2})=\emptyset$ (and hence $\cF_{4,b}(b,\frac{ab}{\xi^2}) \neq\emptyset$), then we choose $(g_0,h_0,n_0)$ to be the triple satisfying
\begin{equation}\label{sub2}
g_0^2-h_0^2b=n_0^2\frac{ab}{\xi^2}
\end{equation}
using the same procedure as above.
Now we finish the proof of Proposition \ref{cond}.
\end{proof}

\begin{proof}[Proof of Lemma \ref{ab}]
We know from algebraic number theory that the ring of integers $\cO_K$ is a free $\Z$-module with rank 4, and that $\text{span}_\Z \{1,\sqrt{a},\sqrt{g+h\sqrt{a}},\sqrt{ag+ah\sqrt{a}}\}$ is a sublattice of $\cO_K$. Therefore we have 
\begin{equation}\label{heng}
|\text{disc}(K)|=|\text{disc}(\cO_K)| \le |\text{disc}(\text{span}_\Z \{1,\sqrt{a},\sqrt{g+h\sqrt{a}},\sqrt{ag+ah\sqrt{a}}\})|=|256a^2(g^2-h^2a)|. 
\end{equation}
Note that $(g,h,n)$ satisfies relation (\ref{n2b}), 
hence $D_K \le 256|a|^2|b|n^2$. By choosing $C_{a,b}=256|a|^3|b|^3$ we have $D_K \le C_{a,b}n^2$. In the same way, the result $D_K \le C_{a,b}n^2$ holds for fields in other five subfamilies
with the same constant $C_{a,b}$. Lemma \ref{ab} then follows. 
\end{proof}

Note that Huard, Spearman, and Williams \cite{HSW91} compute explicitly the discriminant of a quartic field of the form $\Q(\sqrt{g+h\sqrt{a}})$; see Theorem 1 of \cite{HSW91}. We do not approach their method here, since Lemma \ref{ab} is sufficient for our purpose.

Now we prove Proposition \ref{comb}.
Remember that we have fixed the ordered pair $(a,b)$ (hence have fixed $(g_0,h_0,n_0)$). 

\begin{proof}[Proof of (1) of Proposition \ref{comb}]
For $m \in M_a(a,b;X)$, the proof of $K_{[m]} \in \cF_{4,a}(a,b)$ is the same as that of $K_0 \in \cF_{4,a}(a,b)$ in the proof of Proposition \ref{cond}, so we omit it here.

By (\ref{defp}), we know that if $m \in M_a(a,b;X)$, then $n_0^2C_{a,b}m^2 \le n_0^2C_{a,b}\cdot \frac{1}{n_0^2C_{a,b}}X=X$. For any $K_{[m]} \in T_a(a,b;X)$, we have $m \in M_a(a,b;X)$. By (\ref{heng}) we have
\begin{equation*}
D_{K_{[m]}} \le |256a^2((g_0m)^2-(h_0m)^2a)| = 256|a|^2n_0^2|ab|m^2  \le n_0^2C_{a,b}m^2 \le X.
\end{equation*}
Thus, we have $K_{[m]} \in \cF_{4,a}(a,b;X)$. It follows that $T_a(a,b;X) \subset \cF_{4,a}(a,b;X)$.
\end{proof}

\begin{proof}[Proof of (2) of Proposition \ref{comb}]
We let $m_1,m_2 \in M_a(a,b;X)$ such that $m_1 \neq m_2$. If $K_{[m_1]}=K_{[m_2]} \in T_a(a,b;X)$, then $\frac{\sqrt{g_0m_2+h_0m_2\sqrt{a}}}{\sqrt{g_0m_1+h_0m_1\sqrt{a}}}=\sqrt{\frac{m_2}{m_1}} \in K_{[m_1]}$. Thus we have $\sqrt{m_1m_2} \in K_{[m_1]}$. Since $(\sqrt{m_1m_2})^2 \in \Q$ and note that $\Q(\sqrt{a})$ is the only quadratic subfield of $K_{[m_1]}$, we know that $\sqrt{m_1m_2} \in \Q(\sqrt{a})$. Note also that $a$ and $m_1m_2$ are integers in $\Q(\sqrt{a})$, where $a$ is square-free and $m_1m_2$ is not a perfect square, it follows that $|a|$ divides $m_1m_2$. This contradicts the fact that $(|a|,m_1)=(|a|,m_2)=1$ if $a \neq -1$. If $a=-1$, then since $m_1m_2$ is positive and is not a perfect square, $\sqrt{m_1m_2} \notin \Q(\sqrt{-1})$. We also have a contradiction. Therefore we know that $K_{[m_1]} \neq K_{[m_2]}$.
\end{proof}

\begin{proof}[Proof of (3) of Proposition \ref{comb}]
By (2) of Proposition \ref{comb}, there is a one-to-one correspondence between the elements in $T_a(a,b;X)$ and $M_a(a,b;X)$, so $|T_a(a,b;X)|=|M_a(a,b;X)|$. We need to prove that 
\begin{equation}\label{mx}
|M_a(a,b;X)| \gg_{a,b} X^{1/2}.
\end{equation}
For $Z \in \R_{>0}$ and a positive integer $q \ge 2$, we define the set
\begin{equation*}
M(Z,q)=\{ m \in \Z_{>0} \text{ square-free} \ | \ \mathrm{gcd}(m,q)=1, m \le Z \}.
\end{equation*}
It suffices to prove that
\begin{equation}\label{mz}
M(Z,q) \gg_q Z \ \ \text{as} \ Z \to \infty.
\end{equation} 
With (\ref{mz}) in hand, (\ref{mx}) follows immediately once we take $Z=\frac{1}{16n_0\sqrt{|a|^3|b|^3}} X^{1/2}$ and $q=|ab|$.

We recall the M\"{o}bius function $\mu$ and Euler's totient function $\phi$. Also, we temporarily use the notation $(\cdot , \cdot)$ instead of $\mathrm{gcd}(\cdot , \cdot)$ for brevity. Then we have
\begin{eqnarray}
\nonumber M(Z,q) &=& \sum_{\substack{m \le Z \\ (m,q)=1}} \sum_{d^2|m} \mu(d) \ \ = \sum_{d \le \sqrt{Z}} \mu(d) \sum_{\substack{m \le Z \\ (m,q)=1 \\ d^2|m}} 1 \ \ = \sum_{\substack{d \le \sqrt{Z} \\ (d,q)=1}} \mu(d) \sum_{\substack{m \le Z \\ (m,q)=1 \\ d^2|m}} 1 \\
\nonumber &\stackrel{m=m'd^2}{=}& \sum_{\substack{d \le \sqrt{Z} \\ (d,q)=1}} \mu(d) \sum_{\substack{m' \le \lfloor \frac{Z}{d^2} \rfloor \\ (m',q)=1 }} 1 \ \ = \sum_{\substack{d \le \sqrt{Z} \\ (d,q)=1}} \mu(d) \left[ \frac{\phi(q)}{q} \lfloor \frac{Z}{d^2} \rfloor+O(1) \right] \\
\nonumber &=& \frac{\phi(q)}{q} \sum_{\substack{d \le \sqrt{Z} \\ (d,q)=1}} \mu(d) \lfloor \frac{Z}{d^2} \rfloor+O(\sqrt{Z}) \ \ = \frac{Z\phi(q)}{q} \sum_{\substack{d \le \sqrt{Z} \\ (d,q)=1}} \frac{\mu(d)}{d^2}+O(\sqrt{Z})+O(\sqrt{Z}) \\
\nonumber &=& \frac{Z\phi(q)}{q} \sum_{(d,q)=1} \frac{\mu(d)}{d^2}+O(Z\sum_{d>\sqrt{Z}} \frac{1}{d^2})+O(\sqrt{Z}) \\
\nonumber &=& \frac{Z\phi(q)}{q} \prod_{\substack{p \text{ prime} \\ p \nmid q}} (1-p^{-2})+O(\sqrt{Z}) \ \ = \frac{\phi(q)}{q\zeta(2)} \prod_{\substack{p \text{ prime} \\ p \mid q}} (1-p^{-2})^{-1} Z+O(\sqrt{Z}).
\end{eqnarray}
Therefore, (\ref{mz}) follows. Now we finish the proof of Proposition \ref{comb}.
\end{proof}

Another natural way to count $D_4$-quartic fields is to count the fields up to isomorphism. In that way, our lower bound still holds, since one isomorphism class of fields is in one-to-one correspondence with two fields in $\bar{\Q}$. In detail, in the lattice of fields, inside $\bar{\Q}$, $K_1$ and $K_2$ are the only two representatives of the same isomorphism class of fields. In other words, if $K=\Q(\sqrt{g+h\sqrt{a}})$ is a $D_4$-quartic field, where $g \in \Z$, $h \in \Z^*$, then the only other field isomorphic to $K$ in $\bar{\Q}$ is $\Q(\sqrt{g-h\sqrt{a}})$.

\end{section}

\begin{section}{Proof of Theorem \ref{cdtzf}}

In this section we assume that for a field $K \in \cF_4(Q;X)$ with $Q=\Q(\sqrt{a},\sqrt{b})$, $\zeta_{\widetilde{K}}(s)/\zeta_Q(s)=L^2(s,\rho_{\widetilde{K}})$ is zero-free in the region (\ref{azfr2}), and then derive Theorem \ref{cdtzf} with the assumption above. We proceed via a simple adaptation of the argument in \cite{PTW17}, except now we use the fact that the $L$-functions $L(s,\chi_{a^\ast})$, $L(s,\chi_{b^\ast})$, and $L(s,\chi_{({\frac{ab}{\xi^2}})^\ast})$ are fixed, and we move to the right of any exceptional zero they may possess.

First, we consider the zero-free region for $\zeta_Q(s)/\zeta(s)$, where $Q=\Q(\sqrt{a},\sqrt{b})$; this is a product of Dirichlet $L$-functions.  
Theorem 5.26 of \cite{IK04} provides the standard zero-free region for a Dirichlet $L$-function.
\begin{prop}\label{ik}
There exists an absolute constant $C_0>0$ such that for any primitive Dirichlet character $\chi$ modulo $q$, $L(s,\chi)$ has at most one zero $s=\sigma+it$ in the region 
\begin{equation*}
\sigma \ge 1-\frac{C_0}{\log q(|t|+3)}.
\end{equation*}
The exceptional zero may occur only if $\chi$ is real, and it is then a simple real zero, say $\beta_\chi$, with
\begin{equation*}
1-\frac{C_0}{\log 3q} \le \beta_\chi <1.
\end{equation*}
\end{prop}

We set 
\begin{equation*}
q_{\max}=\max\{|a^\ast|,|b^\ast|,|(\frac{ab}{\xi^2})^\ast|\}
\end{equation*}
 and denote 
 \begin{equation}\label{betamax}
 \beta_{\max}=\max \{ \beta_{\chi_{a^\ast}},\beta_{\chi_{b^\ast}},\beta_{\chi_{{(\frac{ab}{\xi^2})}^\ast}} \}.
 \end{equation}
 If none of $\beta_{\chi_{a^\ast}},\beta_{\chi_{b^\ast}},$ or $\beta_{\chi_{{(\frac{ab}{\xi^2})}^\ast}}$ exists, we simply set $\beta_{\max}=\frac34$. Then by Proposition \ref{ik}, $\zeta_Q(s)/\zeta(s)$ has no zeros in the region $R_1 \cap R_2$, where
\begin{equation*}
R_1=\{ \sigma+it: \beta_{\max} < \sigma \le 1 \},
\end{equation*}
\begin{equation*}
R_2=\{ \sigma+it: \sigma \ge 1-\frac{C_0}{\log q_{\max}(|t|+3)} \}.
\end{equation*}

Second, by the hypothesis of Theorem \ref{cdtzf}, the function $\zeta_{\widetilde{K}}(s)/\zeta_Q(s)$ is zero-free in the region (\ref{azfr2}).
Given $\ep_0$, we let $\delta$ be the constant in (\ref{delta}),
so $\delta$ depends only on $\ep_0$.
We choose $\ep_0$ sufficiently small such that $\delta<1-\beta_{\max}$.

Consequently, 
the function $\zeta_{\widetilde{K}}(s)/\zeta(s)$ is zero-free in the region $R_1 \cap R_2 \cap R_3$, where 
\begin{equation*}
R_3=\{ \sigma+it: \sigma \ge 1-\delta, |t| \le (\log D_{\widetilde{K}})^{2/\delta} \}.
\end{equation*}
Since $\beta_{\max}<1-\delta$, the zero-free region $R_1 \cap R_2 \cap R_3$ is the same as
\begin{equation*}
R_4=R_2 \cap R_3=\{ \sigma+it: \sigma \ge \max\{ 1-\frac{C_0}{\log q_{\max}(|t|+3)},1-\delta \}, |t| \le (\log D_{\widetilde{K}})^{2/\delta} \}.
\end{equation*}
In our Chebotarev density theorem we are interested in the range where $D_{\widetilde{K}} \to \infty$, so we can assume that $D_{\widetilde{K}}$ is sufficiently large and the above zero-free region $R_4$ becomes
\begin{equation}\label{7z}
\begin{cases}
& \sigma \ge 1-\delta, \text{ if } |t| \le T_0, \\
& \sigma \ge 1-\frac{C_0}{\log q_{\max}(|t|+3)}, \text{ if } T_0 \le |t| \le (\log D_{\widetilde{K}})^{2/\delta},
\end{cases}
\end{equation}
where $T_0=\frac{e^{C_0/\delta}}{q_{\max}}-3$ is the height of the intersection point of the boundary lines of two zero-free regions for $\zeta_{\widetilde{K}}(s)/\zeta_Q(s)$ and $\zeta_{Q}(s)/\zeta(s)$. In fact we can let 
\begin{equation}\label{c0}
D_{\widetilde{K}} \ge \exp(\exp(\frac{C_0}{2}))
\end{equation}
to fulfill our assumption.

To prove Theorem \ref{cdtzf}, we consider two different ranges of $x$. For $x \ge \exp(80(\log D_{\widetilde{K}})^2)$, we note that the error term allowed in Theorem \ref{cdtzf} is larger than the error term $c_3x \exp(-c_4(\log x)^{1/2})$ (where $c_3,c_4$ are effectively computable constants) in the unconditional effective Chebotarev density theorem of Lagarias and Odlyzko (Theorem 1.3 in \cite{LO75}), so our Chebotarev density theorem holds for such $x$. Now we assume that $x \le \exp(80(\log D_{\widetilde{K}})^2)$.

For $K \in \cF_4(Q)$, we define the weighted prime-counting function as
\begin{equation*}
\psi_\cC(x, \widetilde{K}/\Q)=\sum_{\substack{p,\hspace{.5mm} m \\ p \text{ unramified in } \cO_{\widetilde{K}} \\ \mathrm{Nm}_{\widetilde{K}/\Q}p^m \le x \\ \left[ \frac{\widetilde{K}/\Q}{p} \right]=\cC}} \log p,
\end{equation*}
and the final result for $\pi_\cC(x,\widetilde{K}/\Q)$ will follow from partial summation. By Theorem 7.1 of \cite{LO75}, we have
\begin{equation}\label{ertm}
|\psi_\cC(x, \widetilde{K}/\Q)-\frac{|\cC|}{|G|}x| \le C_1(S(x,T)+E_1+E_2),
\end{equation}
where $C_1$ is an absolute constant and
\begin{equation*}
E_1=\frac{|\cC|}{|G|} (xT^{-1}\log x \log D_{\widetilde{K}}+\log D_{\widetilde{K}}+8\log x+8xT^{-1}\log x \log T),
\end{equation*}
\begin{equation*} 
E_2=\frac{|\cC|}{|G|} (\log x \log D_{\widetilde{K}}+8xT^{-1}(\log x)^2).
\end{equation*}
By (4.24) of \cite{PTW17}, in the case of $G=D_4$, we know that 
\begin{equation}\label{sxt}
|S(x,T)| \le C_2 \frac{|\cC|}{|G|} (E_3+E_4+E_5),
\end{equation}
where $C_2$ is an absolute constant and $E_3=8x^{1/2}(\log D_{\widetilde{K}})^2$, $E_4=x^{1-\delta \log T \log(D_{\widetilde{K}}T^8)}$, $E_5=x^{1-\frac{C}{\log q_{\max}(T+3)}} \log T \log(D_{\widetilde{K}}T^8)$. We set
\begin{equation*}
T=(\log D_{\widetilde{K}})^{2/\delta}.
\end{equation*}

We are able to use the analysis of the error terms $E_1,E_2,E_3$, and $E_4$ in Section 4 of \cite{PTW17} to show that the absolute values of the four error terms are bounded by $C_3 |\cC||G|^{-1} x(\log x)^{-1}$, provided that 
\begin{equation}\label{c4}
D_{\widetilde{K}} \ge C_4.
\end{equation}
Note that $C_3$ and $C_4$ are absolute constants.
The only difference is the term $E_5$ due to a different value of $\cL(T)$, the width of the zero-free region up to the height $T$. In our setting, 
\begin{equation*}
\cL(T)=\frac{C_0}{\log q_{\max}(T+3)}.
\end{equation*}
 
In order to have the bound for the error term as claimed in (\ref{bnd}), we want
\begin{equation*}\label{middle}
x^{1-\cL(T)}\log T\log(D_{\widetilde{K}}T^8) \le C_5\frac{|\cC|}{|G|}x(\log x)^{-1},
\end{equation*}
where $C_5$ is an absolute constant.
If this holds,
the error term in (\ref{ertm}) becomes the right hand side of (\ref{bnd}) after partial summation.
Upon recalling $x \le \exp(80(\log D_{\widetilde{K}})^2)$, it suffices to have
\begin{equation*}
x \ge \exp\{ C_0^{-1} \log[2q_{\max}(\log D_{\widetilde{K}})^{2/\delta}] \log[C_6 \delta^{-2} (\log D_{\widetilde{K}})^4] \},
\end{equation*}
where $C_6=21760C_5^{-1}$.
We write this as
\begin{equation}\label{e5}
x \ge \exp\{ 8C_0^{-1}\delta^{-1} \log \log(D_{\widetilde{K}}^{(2q_{\max})^{\delta/2}}) \log \log(D_{\widetilde{K}}^{{C_6}^{1/4}\delta^{-1/2}}) \}.
\end{equation}

We combine the analysis of $E_5$ with the analysis of error terms $E_1,E_2,E_3,E_4$ and recall (\ref{c0}) and (\ref{c4}). Then we obtain the followings. If 
\begin{equation}\label{c7}
D_{\widetilde{K}} \ge C_7=\max \{\exp(\exp(\frac{C_0}{2})),C_4\},
\end{equation}
then (\ref{sxt}) holds for all 
\begin{equation*}
\kappa_1^{''} \exp{[\kappa_2^{''}(\log \log(D_{\widetilde{K}}^{\kappa_3^{''}}))^{2}]} \le x \le \exp(80(\log D_{\widetilde{K}})^2)
\end{equation*}
with
\begin{equation*}
\kappa_1^{''}=C_6^{1/\delta} \delta^{-2/\delta},
\end{equation*}
\begin{equation*}
\kappa_2^{''}=\max\{4\delta^{-1},8C_0^{-1}\delta^{-1}\},
\end{equation*}
\begin{equation*}
\kappa_3^{''}=\max\{2q_{\max}, {C_6}^{1/4}\delta^{-1/2}\},
\end{equation*}
as a result of (4.30) of \cite{PTW17} and our (\ref{e5}). Moreover, Theorem \ref{cdtzf} holds with
\begin{equation}\label{kappa1}
\kappa_1=40C_6^{1/\delta} \delta^{-2/\delta},
\end{equation}
\begin{equation}\label{kappa2}
\kappa_2=\max\{4\delta^{-1},8C_0^{-1}\delta^{-1}\}+4,
\end{equation}
\begin{equation}\label{kappa3}
\kappa_3=(480C_1)^{1/5}\max\{2q_{\max}, {C_6}^{1/4}\delta^{-1/2}\},
\end{equation}
as a result of (4.47) of \cite{PTW17}.
Note that $\delta$ is given in terms of $\ep_0$ in (\ref{delta}). Aside from absolute constants, $\kappa_i$ depends on $a,b,\ep_0$, since $q_{\max}$ depends only on $a,b$.

\end{section}

\begin{section}{Proof of Theorem \ref{qwe}}\label{pfmain}

We prove Theorem \ref{qwe} via an adaptation of the argument in Section 5,6 of \cite{PTW17}. Notice here that because we have defined our family so that only one factor in (\ref{l}) is varying as $K$ varies, one avoids the difficulties faced in \cite{PTW17} when applying Theorem 2 of \cite{KM02} to noncuspidal representations.

For $Q=\Q(\sqrt{a},\sqrt{b})$ with $\cF_4(Q) \neq \emptyset$,
and for every $X \ge 1$, we define $\widetilde{\cF}_4(Q;X)$ to be the set containing all the Galois closures of $K$ as $K$ varies in $\cF_4(Q;X)$. Moreover, we define
\begin{equation*}
\cL_{Q}(X)=\{ L(s,\rho_{\widetilde{K}}): \widetilde{K} \in \widetilde{\cF}_4(Q;X) \},
\end{equation*}
where we recall $\rho_{\widetilde{K}}$ is the faithful 2-dimensional irreducible representation of $D_4$.
If $K_1,K_2 \in \cF_4(Q)$ have the property that $L(s,\rho_{\widetilde{K_1}})=L(s,\rho_{\widetilde{K_2}})$, then $\widetilde{K_1}=\widetilde{K_2}$ by \cite[Proposition 6.3]{PTW17}. Importantly in this application, we note that the character $\rho_{\widetilde{K}}$ is faithful. Therefore, $\cL_{Q}(X)$ is a set and that the elements in $\cL_{Q}(X)$ are in one-to-one correspondence with those in $\widetilde{\cF}_4(Q;X)$.


Recall from the lattice of fields in Section \ref{mot} that four $D_4$-quartic fields share one Galois closure, hence share one $L$-factor $L(s,\rho_{\widetilde{K}})$. In order to prove Theorem \ref{qwe}, it suffices to prove the following theorem.

\begin{thm}\label{ch4}
Let $Q=\Q(\sqrt{a},\sqrt{b})$ satisfy $\cF_4(Q) \neq \emptyset$.
For every $0<\ep_0<\frac14$, in the set $\cL_{Q}(X)$,
there are $\ll_{\ep_0} X^{\ep_0}$ $L$-functions $L(s,\rho_{\widetilde{K}})$ that could have a zero in the region (\ref{azfr2}).
\end{thm}

\begin{proof}
We will prove this via an application of Theorem 2 of \cite{KM02}, which gives an upper bound for the zero density in a family of cuspidal automorphic $L$-functions. We first verify the conditions Kowalski and Michel's work requires in our specific setting.

We note that the strong Artin conjecture is true for dihedral groups (see \cite{Lan80}). Thus, for a $D_4$-quartic field $K$ and its associated $L$-function $L(s,\rho_{\widetilde{K}})$, there exists an automorphic representation $\pi_{\widetilde{K}}=\pi(\rho_{\widetilde{K}})$ on $\mathrm{GL}_2(\mathbb{A}_\Q)$ such that $L(s,\rho_{\widetilde{K}})$ and $L(s,\pi_{\widetilde{K}})$ agree almost everywhere. Since $\rho_{\widetilde{K}}$ is irreducible, $\pi_{\widetilde{K}}$ is cuspidal. Moreover, if $\pi$ is cuspidal and $L(s,\pi_v)=L(s,\rho_v)$ for almost all $v$, then in fact $L(s,\pi)=L(s,\rho)$ (see Proposition 2.1 of \cite{Mar03}). Thus, there exists a cuspidal automorphic representation $\pi_{\widetilde{K}}$ on $\GL_2(\mathbb{A}_\Q)$ such that $L(s,\pi_{\widetilde{K}})=L(s,\rho_{\widetilde{K}})$. We let
\begin{equation*}
S_{Q}(X)=\{ \pi_{\widetilde{K}}: \widetilde{K} \in \widetilde{\cF}_4(Q;X) \}.
\end{equation*}
Since $\cL_{Q}(X)$ is a set, so is $S_{Q}(X)$. Moreover, if $\pi_{\widetilde{K_1}}=\pi_{\widetilde{K_2}}$, then $L(s,\rho_{\widetilde{K_1}})=L(s,\rho_{\widetilde{K_2}})$ and hence $\widetilde{K_1}=\widetilde{K_2}$. Therefore, the elements in $S_{Q}(X)$ are in one-to-one correspondence with those in $\widetilde{\cF}_4(Q;X)$.

The result of Theorem 2 of \cite{KM02} requires four conditions on $S_{Q}(X)$, which we now verify.

(1) Every element in $S_{Q}(X)$ satisfies the Ramanujan-Petersson Conjecture, since the Ramanujan-Petersson Conjecture is automatically true for automorphic $L$-functions corresponding to Artin $L$-functions. 

(2) There exists $A>0$ such that for all $X \ge 1$ and all $\pi \in S_{Q}(X)$,
\begin{equation*}
 \mathrm{Cond}(\pi) \ll X^{A}. 
 \end{equation*}
Indeed, Lemma 6.1 of \cite{PTW17} shows that $D_{\widetilde{K}} \ll D_K^4$, so that by the conductor-discriminant formula $D_{\widetilde{K}}=\prod_{\rho \in \mathrm{Irr}(D_4)} \mathrm{Cond}(\rho)^{\rho(1)}$, we see that $A=4$ suffices for our purpose. 

(3) For all $X \ge 1$, we have
\begin{equation*}\label{pr3}
 |S_{Q}(X)| \ll X. 
 \end{equation*}
This holds because $|S_{Q}(X)|=|\widetilde{\cF}_4(Q;X)| \le |\cF_4(Q;X)| \ll X$.

(4) Since the strong Artin conjecture is true for $G=D_4$, for every $K \in \cF_4(Q;X)$, the convexity bound
\begin{equation*}
|L(s,\pi_{\widetilde{K}})| \ll_{\ep} (\mathrm{Cond}(\pi_{\widetilde{K}})(|t|+2)^m)^{(1-\Re(s))/2+\ep}
\end{equation*}
holds for any $0 \le \Re(s) \le 1$ and any $\ep>0$. Also, we have an analogous convexity bound for Rankin-Selberg $L$-functions $L(s,\pi_{\widetilde{K}} \otimes \pi_{\widetilde{K^{'}}})$, where $\pi_{\widetilde{K}} \ncong \pi_{\widetilde{K^{'}}}$. See Section 6.2 of \cite{PTW17}.

We define the zero-counting function for $\pi=\pi_{\widetilde{K}}$:
\begin{equation*}
N(\pi;\alpha,T)=| \{ s=\beta+i\gamma: \beta \ge \alpha, |\gamma| \le T, L(s,\pi)=0\} |.
\end{equation*}
Now we apply Theorem 2 in \cite{KM02}, which in our context takes the following form.\begin{thm}\label{km}
Let $\alpha \ge \frac34$ and $T \ge 2$. In the context of $S_{Q}(X)$ above, there exists a constant $B \ge 0$, depending only on the parameters $(a,b,A)$ in the four properties above, such that for every $c_0>21$, we have that there exists a constant $M_{c_0}$ depending only on $c_0$ such that for all $X \ge 1$,
\begin{equation*}
\sum_{\pi \in S_{a,b}(X)} N(\pi;\alpha,T) \le M_{c_0} T^B X^{c_0\frac{1-\alpha}{2\alpha-1}}.
\end{equation*}
\end{thm}
To apply this, letting $0<\ep_0<\frac14$ be given, we set $c_0=21+\ep_0$ and set $\alpha,T$ such that $\frac{c_0(1-\alpha)}{2\alpha-1}=\ep_0/2$, $T=X^{\ep_0/(2B)}$ and recall (\ref{delta}), the defining formula for $\delta$. So we have $\alpha=\frac{2c_0+\ep_0}{2(c_0+\ep_0)}$ and $\delta=1-\alpha$. Theorem \ref{km} shows that there are $ \ll_{\ep_0} X^{\ep_0}$ $L$-functions in $\cL_{Q}(X)$ that could have a zero in $R(X)=\{ s=\sigma+it: 1-\delta \le \sigma \le 1, |t| \le X^{\ep_0/(2B)}\}$. Consequently, aside from $\ll_{\ep_0} X^{\ep_0}$ exceptions in $\cL_{Q}(X)$, $L(s,\rho_{\widetilde{K}})=L(s,\pi_{\widetilde{K}})$ is zero-free in the region (\ref{azfr2}) for all $X$ sufficiently large such that
\begin{equation}\label{last}
(\log D_{\widetilde{K}})^{2/\delta}<X^{\ep_0/(2B)}.
\end{equation}
Note that by Lemma 6.1 of \cite{PTW17}, we have $D_{\widetilde{K}} \le c_1 D_K^4$, for some absolute constant $c_1>0$. Together with the relation $D_K \le X$, (\ref{last}) will follow as long as $X$ is sufficiently large such that
\begin{equation}\label{last2}
(4\log X+\log c_1)^{2/\delta}<X^{\ep_0/(2B)}.
\end{equation}
Any fixed power of $X$ is greater than any fixed power of $\log X$ once $X$ is sufficiently large. Therefore, there exists a constant $D_0=D_0(\ep_0)$ such that (\ref{last2}) (hence (\ref{last})) holds whenever $X \ge D_0$. For the remaining cases with small discriminant $X<D_0$, we have $|\cL_{Q}(X)| \le |\cF_4(Q;X)| \ll D_0 \ll_{\ep_0} 1$. Theorem \ref{ch4} then follows.
\end{proof}

\end{section}

\begin{section}{Proof of Theorem \ref{tor}}

We use Theorems \ref{cdtzf} and \ref{qwe} to prove Theorem \ref{tor}. As a consequence of Theorem \ref{cdtzf}, we have the following proposition (analogous to Corollary 3.16 in \cite{PTW17}).
\begin{prop}\label{smallprime}
Let 
$Q=\Q(\sqrt{a},\sqrt{b})$ be a biquadratic field satisfying $\cF_4(Q) \neq \emptyset$. For every $\ep_0>0$ sufficiently small, let $\delta$ be defined as in (\ref{delta}). 
Then for any $\sigma>0$, there exists a constant $B_3=B_3(\ep_0,\sigma)$ such that for every $X \ge 1$, every field $K \in \cF_4(Q)$ that has $D_K \ge B_3$ and whose associated $L$-function $L(s,\rho_{\widetilde{K}})$ (see (\ref{l})) is zero-free in the region (\ref{azfr2}), has the property that for any fixed conjugacy class $\cC$ in $D_4$, there are at least 
\begin{equation*}
\gg_{\cC,\sigma} \frac{D_K^\sigma}{\log D_K}
\end{equation*}
rational primes $p \le D_K^\sigma$ with Artin symbol $\left[ \frac{\widetilde{K}/\Q}{p} \right]=\cC$.
\end{prop}
Proposition \ref{smallprime} is deduced from Theorem \ref{cdtzf} in the same manner that Corollary 3.16 is deduced from Theorem 3.1 in Section 6.6 of \cite{PTW17}, thus we omit the proof here.

We recall Lemma 2.3 of Ellenberg and Venkatesh \cite{EV07}. In our setting, it has the following form.
\begin{prop}\label{ev}
Let $K$ be a $D_4$-quartic field and fix a positive integer $\ell$. Set $\eta<\frac{1}{6\ell}$ and suppose that there are at least $M$ rational primes with $p \le D_K^\eta$ that are unramified and split completely in $K$. Then
\begin{equation*}
|\Cl_K[\ell]| \ll_{\ell,\ep_1} D_K^{\frac12+\ep_1}M^{-1},
\end{equation*}
for every $\ep_1>0$.
\end{prop}

Now we deduce Theorem \ref{tor} from Theorem \ref{qwe}, Propositions \ref{smallprime} and \ref{ev}, recalling the numbers $\ep_0$ and $\delta$ chosen in Proposition \ref{smallprime}. We set $\cC=\Id$, so we count unramified primes which split completely in $\widetilde{K}$, hence split completely in $K$. For any positive integer $\ell$, we choose $\ep_2>0$ sufficiently small and set $\sigma=\frac{1}{6\ell}-\ep_2$. Then for every $X \ge 1$, for any field $K \in \cF_4(Q;X)$ with $D_K \ge B_3$ that is not one of $\ll_{\ep_0} X^{\ep_0}$ fields whose associated $L$-function $L(s,\rho_{\widetilde{K}})$ could have a zero in the region (\ref{azfr2}), there are $\gg_{\ell,\ep_2} D_K^{\frac{1}{6\ell}-\ep_2}/\log D_K$ primes $p \le D_K^{\frac{1}{6\ell}-\ep_2}$ that split completely in $K$. Thus, by Proposition \ref{ev}, for such a field $K$, we have
\begin{equation*}
|\Cl_K[\ell]| \ll_{\ell,\ep_1,\ep_2} D_K^{\frac12-\frac{1}{6\ell}+\ep_2+\ep_1},
\end{equation*}
for all $\ep_1>0,\ep_2>0$ sufficiently small. Note that the number of $K \in \cF_4(Q;X)$ such that $D_K<B_3$ is $\ll B_3$, which is a constant depending on $\ell,\ep_0,\ep_2$. Theorem \ref{tor} then follows as we choose $\ep_0,\ep_1,\ep_2$ such that $\max\{\ep_0,\ep_1+\ep_2\} \le \ep$ and 
$\delta=\frac{\ep_0}{42+4\ep_0}<1-\beta_{\max}$. Note that $\beta_{\max} \in (0,1)$, defined in (\ref{betamax}), is a fixed number when $Q=\Q(\sqrt{a},\sqrt{b})$ is fixed.

\end{section}

\section*{Acknowledgement}

The author thanks his advisor Lillian Pierce for her consistent guidance, Leslie Saper for his lectures on algebraic number theory, Jiuya Wang and Asif Zaman for communication during this research. In addition, the author thanks the Hausdorff Center for Mathematics and the Institute for Advanced Study for hosting visits during this research. Finally, the author thanks the referee for providing valuable comments in mathematics and in exposition, especially for the comment on suggesting the notion of extended quadratic fields.

\footnotesize{}

\bibliographystyle{alpha}

\bibliography{NoThBibliography}

  \textsc{Department of Mathematics, Duke University, 120 Science Drive, Durham NC 27708 USA} \par  
  \textit{E-mail address}: \texttt{chen.an@duke.edu} \par
  \addvspace{\medskipamount}


\end{document}